\documentclass[12pt]{amsart}
\usepackage{amsfonts}
\usepackage{amsmath,amscd, amsthm, tabularx}
\newtheorem{prop}{Proposition}[section]
\newtheorem{thm}[prop]{Theorem}

\newtheorem{lemma}[prop]{Lemma}
\newtheorem{corollary}[prop]{Corollary}
\newtheorem{rmk}[prop]{Remark}
\theoremstyle{definition}
\newtheorem{definition}[prop]{Definition}
\usepackage{cite,graphicx,color}
\def\Hom{{\rm Hom}}
\def\Cone{{\rm Cone}}
\def\Conv{{\rm Conv}}

\title{Generalized compactifications of Batyrev hypersurface families}
\author{Karl Fredrickson}
\address{8307 22nd Ave. NW \\
Seattle, WA, 98117}
\email{karlfredrickson@gmail.com}
\date{}

\begin{document}
\maketitle
\begin{abstract}We show how Calabi-Yau hypersurface families arising from Batyrev's construction can be resolved and compactified using a type of fan more general than an MPCP resolution.  This can lead to smooth projective compactifications that are not obtainable from the original construction.  In the threefold case, we show that generic members of the resulting family are always smooth.
\end{abstract}

\section{Introduction}

Batyrev's fundamental construction of Calabi-Yau hypersurfaces in \cite{Ba} provides both a large source of Calabi-Yau families and a very explicit construction of mirror symmetry for these families. Let us recall the basic setup. Let $N \cong \mathbb{Z}^d$ for some $d$ be a lattice, and consider the corresponding real vector space $N_\mathbb{R} = N \otimes \mathbb{R}$, the dual lattice $M = \Hom(N, \mathbb{Z})$, and the dual vector space $M_\mathbb{R} = M \otimes \mathbb{R}$.  Also let $\langle, \rangle$ be the natural real-valued dual pairing between $M_\mathbb{R}$ and $N_\mathbb{R}$. 
A lattice polytope in $N_\mathbb{R}$ is the convex hull of finitely many lattice points in $N$. (For an introduction to toric geometry, see the books \cite{CLS} or \cite{fulton}.)

To construct a Calabi-Yau hypersurface family, we first need a reflexive polytope \cite{Ba}.  A lattice polytope $\Delta \subseteq N_\mathbb{R}$ with the origin in its interior is said to be {\em reflexive} if its dual polytope $$\Delta^* = \{ m \in M_\mathbb{R} \ | \ \langle m, n \rangle \geq -1 \hbox{ for all } n \in \Delta \}$$ is also lattice polytope in $M_\mathbb{R}$.  The Calabi-Yau hypersurface family associated to $\Delta$ is a compactification of the hypersurface family in the torus Spec $\mathbb{C}[M] \cong (\mathbb{C}^*)^d$ defined by 
\begin{equation} \label{definingeqn}
\sum_{m \in M \cap \Delta^*} c_m z^m = 0
\end{equation}
where the $c_m \in \mathbb{C}$ are generic coefficients.

To compactify this family, the open torus Spec $\mathbb{C}[M]$ needs to be included into a larger toric variety.  One natural choice is the toric variety $X(\Delta)$ given by the fan $\Sigma(\Delta)$, the fan of cones over proper faces of $\Delta$.  In the case that the toric variety defined by $\Sigma(\Delta)$ is smooth, this yields a family $\mathcal{X}(\Delta) \subseteq X(\Delta)$ of smooth projective Calabi-Yau hypersurfaces.  In general, however, $X(\Delta)$ is not smooth, and to desingularize the Calabi-Yau family as much as possible, we require another important ingredient of Batyrev's construction, an MPCP (maximal projective crepant partial) resolution of $X(\Delta)$.  An {\em MPCP resolution} of the toric variety $X(\Delta)$ is defined by a fan $\widehat{\Sigma}(\Delta)$ which satisfies the following two conditions:
\begin{enumerate}
\item[(i)] $\widehat{\Sigma}(\Delta)$ is complete, simplicial, projective, and a subdivision of $\Sigma(\Delta)$.
\item[(ii)] The set of rays of $\widehat{\Sigma}(\Delta)$ is equal to the set of rays over nonzero lattice points in $\Delta$.
\end{enumerate}

If $\Delta$ is of dimension four or less, then Equation \eqref{definingeqn} will define a family of smooth Calabi-Yau hypersurfaces in the toric variety given by $\widehat{\Sigma}(\Delta)$.  

\bigskip
The purpose of this paper is to study what happens when the hypersurface family defined by Equation \eqref{definingeqn} is compactified and resolved with a type of fan more general than an MPCP resolution, which we call a $\Delta$-maximal fan.  A {\em $\Delta$-maximal fan} has the same definition as an MPCP resolution, only without the requirements that it has to be a subdivision of $\Sigma(\Delta)$ or projective. This paper is devoted to proving smoothness and regularity results for the toric variety $X(\Sigma)$ associated to a $\Delta$-maximal fan $\Sigma$, and for the anticanonical hypersurface family $\mathcal{X}(\Sigma) \subseteq X(\Sigma)$.

These questions have already been studied for a particular example in \cite{fred3}.  It concerned a four-dimensional {\em smooth Fano polytope} $\Delta$ (i.e., a reflexive polytope with the property that the toric variety $X(\Delta)$ is smooth), so a family of smooth Calabi-Yau hypersurfaces $\mathcal{X}(\Delta) \subseteq X(\Delta)$ exists without the need for an MPCP resolution.  Nonetheless, there is another simplicial fan with the same set of rays as $\Sigma(\Delta)$, giving another compactification $\mathcal{X}_{bl}$.  Members of $\mathcal{X}_{bl}$ are smooth projective Calabi-Yau varieties that are topologically distinct from those of $\mathcal{X}(\Delta)$. By further analyzing Calabi-Yau 
varieties with the same Hodge numbers in Kreuzer and Skarke's database \cite{ks}, it was shown that members of $\mathcal{X}_{bl}$ are topologically distinct from all other CY varieties that come from MPCP resolutions of reflexive four-polytopes. The fact that topologically distinct Calabi-Yau families can be produced is one of the reasons why $\Delta$-maximal fans are interesting.  For a further discussion of this example we refer to Section 5.

\bigskip

Let us describe organization and main results of this paper. 

In Section 2 we prove that for a reflexive polytope $\Delta$ and a $\Delta$-maximal fan $\Sigma$ of arbitrary dimension $X(\Sigma)$ always has a singular locus of codimension $\geq 4$ (Proposition~\ref{prop:3dimsmooth}).

Sections 3 and 4 focus on reflexive polytopes of dimension four.  It follows from Proposition~\ref{prop:3dimsmooth} that in this case a $\Delta$-maximal fan $\Sigma$ must be Gorenstein. This means the anticanonical line bundle $\mathcal{L}$ of $X(\Sigma)$ exists and the family $\mathcal{X}(\Sigma)$ can be defined by taking generic global sections of $\mathcal{L}$. We show that generic members of $\mathcal{X}(\Sigma)$ have at worst isolated singularities occuring at zero-dimensional toric strata of $X(\Sigma)$ (Proposition \ref{smoothhypersurface}).  Proposition \ref{smoothness} and Definition \ref{maximalcones} then give a combinatorial criterion (the notion of a reflexive 4-polytope having ``good maximal cones'') that suffices to show generic members of $\mathcal{X}(\Sigma)$ are smooth at these remaining points. 

In section 4, we prove that this combinatorial condition always holds, so that generic members of $\mathcal{X}(\Sigma)$ are smooth for any reflexive 4-polytope $\Delta$ and $\Delta$-maximal fan $\Sigma$ (Theorem \ref{goodness}).  The proof uses the secondary fan associated to the set of lattice points in $\Delta$.  At the beginning of the section we briefly introduce the definitions and facts about the secondary fan and wall crossings that will be needed for the proof.

As mentioned above, Section 5 discusses a particular example of a $\Delta$-maximal fan and the associated family of smooth CY threefolds.  

Section 6 describes how our main result can be applied to the problem of constructing an \emph{extremal transition} between two families of CY threefolds, given a nested pair of reflexive 4-polytopes $\Delta_1 \subseteq \Delta_2$.

\bigskip

It may be useful to briefly explain why analyzing the members of $\mathcal{X}(\Sigma)$ is more difficult in the case of an arbitrary $\Delta$-maximal fan than an MPCP resolution.  In Batyrev's construction generic members of the hypersurface family satisfy a condition called ``$\Delta$-regularity'' (Definition 3.1.1 from \cite{Ba}), which implies that all their singularities are inherited from the singularities of the ambient toric variety.  Thus, if the toric variety is sufficiently smooth (as in the case of an MPCP subdivision of a reflexive four-polytope), generic anticanonical hypersurfaces in the toric variety will be smooth. Because hypersurfaces in a more general $\Delta$-maximal fan need not satisfy $\Delta$-regularity, singularities of $\mathcal{X}(\Sigma)$ could potentially occur even on a smooth open subset of $X(\Sigma)$.

Lastly, it bears mentioning for future work that these results could potentially be applied to the complete intersection Calabi-Yau families studied by Batyrev and Borisov \cite{Borisov,BB}, since resolving such a family also requires an MPCP subdivision of a reflexive polytope. 

\bigskip

\textbf{Acknowledgements.} I would like to thank the organizers of the conference ``Calabi-Yau Manifolds and their Moduli'' for the invitation to speak.  I would also like to thank Benjamin Nill for substantial help with the results in the paper, including the proof of Proposition \ref{prop:3dimsmooth}, as well as help with writing and organization.

\section{$\Delta$-maximal fans and their singular locus}

Let us first fix some notation. Given points $v_1, \dots, v_n$ in a real vector space, we use $\Conv(v_1, \dots, v_n)$ to denote their convex hull.  We also define 
$$\Cone(v_1, \dots, v_n) = \{ r_1 v_1 + \cdots + r_n v_n \ | \ r_1, \dots r_n \in \mathbb{R}_{\geq 0} \},$$ the convex cone generated by $v_1, \dots, v_n$.  A cone is said to be pointed, or strongly convex, if it contains no nontrivial linear subspaces of the ambient vector space.

Given a fan $\Sigma$ in $N_\mathbb{R}$, $X(\Sigma)$ will always denote the toric variety obtained from $\Sigma$.  Where appropriate, $\mathcal{X}(\Sigma)$ will denote the family of generic anticanonical hypersurfaces in $X(\Sigma)$. Given any fan $\Sigma$, we write $\Sigma^{[n]}$ for the subfan consisting of all cones of $\Sigma$ of dimension $n$ or less.

For a facet (maximal proper face) $F$ of a reflexive polytope $\Delta \subseteq N_\mathbb{R}$, we will sometimes make use of the outer normal to $F$, $u_F \in M$.  This is defined as the unique lattice point such that $\langle u_F, x \rangle = 1$ for all $x \in F$.  In terms of the dual polytope $\Delta^* \subseteq M_\mathbb{R}$ defined before, $u_F$ is the negative of the vertex $v \in \Delta^*$ dual to $F$.

\smallskip

Let us define the object of study of this paper.

\begin{definition}  Given a fan $\Sigma$ and a reflexive polytope $\Delta$, we say $\Sigma$ is {\it $\Delta$-maximal} if it satisfies the following:

1. The set of rays of $\Sigma$ is equal to the set of rays over nonzero lattice points in $\Delta$.

2. $\Sigma$ is complete and simplicial.
\end{definition}

The crucial difference between this definition and the definition of an MPCP resolution is that the fan associated to an MPCP resolution must be a subdivision of $\Sigma(\Delta)$, whereas there is no such restriction on a $\Delta$-maximal fan.  Also, note that we do not require $\Sigma$ to be projective. 

\smallskip
Our first result shows that it is possible to generalize well-known statements about MPCP resolutions in \cite{Ba} to this more general situation. Let us recall that a (necessarily simplicial) cone is {\em unimodular}, if it is spanned by a lattice basis. A fan is {\em unimodular}, if all its maximal cones are unimodular (equivalently, the associated toric variety is smooth).

\begin{prop} \label{prop:3dimsmooth}\label{prop:fourdimsmooth} Let $\Delta$ be a reflexive polytope and $\Sigma$ a $\Delta$-maximal fan. 
\begin{enumerate}
\item Any cone of $\Sigma$ of dimension at most $3$ is unimodular.  In particular, 
the toric variety $X(\Sigma)$ has a singular locus of codimension $\geq 4$. 
\item Any cone of $\Sigma$ of dimension $4$ is unimodular, if it is not contained in a maximal cone of $\Sigma(\Delta)$. 
In particular, the toric variety $X(\Sigma^{[4]})$ is Gorenstein.
\end{enumerate}
\end{prop}

\begin{rmk}{\rm Let us note that these results are optimal in the following sense. First, there exist reflexive polytopes $\Delta$ in dimension $4$ such that every lattice point in $\Delta$ is a vertex, and the fan $\Sigma(\Delta)$ is simplicial but not unimodular. Hence, $\Sigma(\Delta)$ is $\Delta$-maximal but not unimodular. Second, it is computationally straightforward to find five linearly independent lattice points in the dual reflexive 5-polytope $\Delta^*$ associated to $\mathbb{P}^5$ such that the cone over these points is not contained in a maximal cone of $\Sigma(\Delta^*)$, it contains no other lattice points of $\Delta^*$, and it is not unimodular. This cone can then be extended to a $\Delta^*$-maximal fan.}
\end{rmk}

\smallskip

\begin{proof} Let us assume there exists a non-unimodular $k$-dimensional cone $C$ in a $\Delta$-maximal fan $\Sigma$ with $2 \le k \le 4$. We may assume 
that $k$ is chosen minimal. Let $C$ be spanned by lattice points $v_1, \ldots, v_k$ in $\partial \Delta$. We consider the lattice parallelepiped
$$P = \left\{ \sum_{i=1}^k c_i v_i \ | \ 0 \leq c_i \leq 1 \text{ for } i=1,\ldots,k\right\}.$$ 
If $C$ is not a unimodular cone, then $P$ must contain a lattice point $x$ which is not an integer linear combination of $v_1, \dots, v_k$, so that $x = \sum_{i=1}^k a_i v_i$, $a_i \in [0,1]$ (for $i=1,\ldots,k$) with not all $a_i$ integers. By the minimality assumption, we have $a_i \not =0$ for all $i=1, \ldots, k$. We may also assume that $a_1+ \cdots + a_k \leq \frac{k}{2}$, since otherwise we can replace $x$ with the lattice point $v_1+\cdots+v_k-x$, also contained in $P$. Since $C$ is a cone in a $\Delta$-maximal fan, $x$ cannot be contained in $\Delta$.

Let $F$ be a maximal face of $\Delta$ such that $x$ is contained in the cone over $F$.  Let $u_F$ be the outer normal to $F$.  Then $2 \le \langle u_F, x \rangle$ since $x \not \in \Delta$ and $\langle u_F, x \rangle$ is an integer. On the other hand, $\langle u_F, x \rangle = \langle u_F, \sum_{i=1}^k a_i v_i \rangle \le a_1+ \cdots +a_k \leq \frac{k}{2},$ because each $\langle u_F, v_i \rangle$ is no more than 1.  This implies $4 \le k$, hence $k=4$ and $\langle u_F, v_i \rangle = 1$ for $i=1,\ldots,k$, so all $v_1, \ldots, v_k$ lie in $F$.
\end{proof}

\section{CY-threefolds coming from $\Delta$-maximal fans}

We will now restrict our focus to the case of four-dimensional reflexive polytopes. It follows from Proposition~\ref{prop:fourdimsmooth} that given a reflexive 4-polytope $\Delta$ and a $\Delta$-maximal fan $\Sigma$, that $X(\Sigma)$ is Gorenstein, and therefore we can define a family of hypersurfaces in $X(\Sigma)$ by taking global sections of the anticanonical bundle.  Because the rays of $\Sigma$ consist of all the rays over nonzero lattice points in $\Delta$, the set of monomial global sections of the anticanonical bundle will be in bijective correspondence with lattice points in $\Delta^*$, just as in the standard Batyrev construction.  Thus in the equation $$\sum_{m \in M \cap \Delta^*} c_m z^m = 0$$ from the introduction, we may consider the left-hand side as a generic global section of the anticanonical bundle of $X(\Sigma)$, defining the family $\mathcal{X}(\Sigma)$. 

We will need the following basic principle concerning lattice points in reflexive polytopes (note that the case of two vertices is contained in \cite[Proposition 4.1]{Nil05}).

\begin{lemma} \label{lemma:commonface} If $v_1, \dots, v_k$ are lattice points in a reflexive polytope $\Delta$, then $v_1+\cdots+v_k$ is contained in $r \cdot \partial \Delta$ for some integer $0 \leq r \leq k$, and $r=k$ if and only if $v_1, \dots, v_k$ are in a common proper face of $\Delta$. \end{lemma}  

\begin{proof}  Let $F$ be a maximal face of $\Delta$ such that $v_1 + \cdots + v_k$ is contained in the cone over $F$.  Then $v_1 + \cdots + v_k$ is in $r F$ for 
$0 \le r := \langle u_F, v_1 + \cdots + v_k \rangle \le \sum_{i=1}^k \langle u_F, v_i \rangle \le k$. Here, the upper equality implies that $v_1, \ldots, v_k$ are in $F$.
\end{proof}

The following two lemmas will be used at later points in proving smoothness.

\begin{lemma} \label{lemma:edges} Let $C = \Cone(v_1,v_2)$ be a two-dimensional cone in a $\Delta$-maximal fan $\Sigma$, where $v_1$ and $v_2$ are lattice points in $\partial \Delta$.  Then the line segment $\Conv(v_1,v_2)$ is contained in $\partial \Delta$. \end{lemma}

\begin{proof}  By Lemma~\ref{lemma:commonface} for $n=2$, either $v_1$ and $v_2$ are contained in a common proper face of $\Delta$, or $v_1+v_2$ is contained in $\Delta$.  In the second case, we would have that the nonzero lattice point $v_1+v_2$ is in the relative interior of $C$, which is not possible since then the ray over $v_1+v_2$ would not be a ray of $\Sigma$, and $\Sigma$ could not be $\Delta$-maximal.  It follows that $v_1$ and $v_2$ are in a common face $f$ of $\Delta$, and therefore $\Conv(v_1,v_2) \subseteq f \subseteq \partial \Delta$. \end{proof}

\begin{lemma} \label{uniqueness} Suppose that $C = \Cone(v_1, v_2, v_3)$ is a cone in a $\Delta$-maximal fan, with $v_1, v_2, v_3$ lattice points in $\partial \Delta$, not contained in a common facet.  Then: 
\begin{enumerate}
\item There exists a vertex $m$ of $\Delta^*$ such that $\langle m, v_i \rangle = \langle m, v_j \rangle = -1$ for some distinct $i, j \in \{1, 2, 3 \}$, and $\langle m, v_k \rangle = 0$ for the remaining $k \in \{1, 2, 3 \}$.
\item For a fixed choice of $i, j, k$ such that $\{ i, j, k \} = \{1, 2, 3 \}$, there is at most one lattice point $m \in \Delta^*$ such that $\langle m, v_i \rangle = \langle m , v_j \rangle = -1$ and $\langle m, v_k \rangle =0$.
\end{enumerate}  
\end{lemma}

\begin{proof} For part 1, by Lemma~\ref{lemma:commonface}, we must have that either $v_1 +v_2+v_3 \in 2 \cdot \partial \Delta$ or $v_1+v_2+v_3 \in \partial \Delta$.  In the latter case, $C$ could not be a cone in a $\Delta$-maximal fan.  Thus we have $v_1+v_2+v_3 \in 2 \cdot \partial \Delta$, and $(v_1+v_2+v_3)/2$ must be contained in some facet of $\Delta$.  Let $m$ be the vertex of $\Delta^*$ dual to this facet.  We must have that $\langle m,(v_1+v_2+v_3)/2 \rangle =-1$, and since $\langle m,v_1 \rangle$, $\langle m,v_2 \rangle$ and $\langle m,v_3 \rangle$ must be integers $\geq -1$, the only way to achieve this is to have $\langle m,v_i \rangle = \langle m,v_j \rangle = -1$ for exactly two $i, j \in \{ 1,2,3 \}$, and $\langle m,v_k \rangle = 0$ for the remaining $k \in \{ 1, 2, 3 \}$.

For part 2, assume without loss that $i = 1$, $j = 2$, $k = 3$.  Since $m$ evaluates to identically $-1$ on $T=\Conv(v_1, v_2, (v_1+v_2+v_3)/2)$, the triangle $T$ must be contained in the boundary of $\Delta$.

We claim that $T$ cannot be contained in a two-dimensional face of $\Delta$.  Suppose by contradiction that $T \subseteq g$ for some two-dimensional face $g \subseteq \partial \Delta$.  Note that $g \cap C$ is a two-dimensional polygon $\Conv(v_1, v_2, w_1, \dots, w_n)$ for some points $w_1, \dots, w_n \in C$.  If we let $A$ be the two-dimensional affine space spanned by $T$, then by calculation, the intersection of $A$ with the relative boundary of $C$ consists of the line segment $\Conv(v_1, v_2)$ and the two rays $v_1+\mathbb{R}_{\geq 0} v_3$, $v_2+\mathbb{R}_{\geq 0} v_3$.  Because $\Conv(v_1, v_3)$ and $\Conv(v_2,v_3)$ are contained in $\partial \Delta$ by Lemma~\ref{lemma:edges}, the only part of this intersection contained in $\partial \Delta$ is $\Conv(v_1, v_2)$.  

Thus $w_1, \dots, w_n$ must be contained in the relative interior of $C$.  Because $C$ is a cone in a $\Delta$-maximal fan, $w_1, \dots, w_n$ cannot be lattice points.  However, since the $w_i$ are in the relative interior of $C$, they would have to be vertices of the entire face $g$.  This is a contradiction since $g$ is a lattice polytope.

It follows that no two distinct lattice points $m, q \in \Delta^*$ can both evaluate to identically $-1$ on $T$, because otherwise $T$ would be contained in the face 
$$\{ d \in \Delta \ | \ \langle m, d \rangle = \langle q, d \rangle = -1 \}$$
of $\Delta$, which would have to be of dimension $\leq 2$.\end{proof}

We will also need to make use of a generic smoothness result, Lemma 3.1 from \cite{fred3}, which can be proven by a short argument with Sard's theorem.  

\begin{lemma} \label{genericsmth} Let $a_0, \dots, a_r \in \mathbb{C}$, and let $$f = a_0 z^{m_0}+\sum_{i=1}^r a_i z^{m_i}$$ be a regular function on $$(\mathbb{C}^*)^s \times \mathbb{C}^t \cong \hbox{{\rm{Spec}}} \ \mathbb{C}[z^{\pm 1}_1, \dots, z^{\pm 1}_s, z_{s+1}, \dots, z_{s+t}],$$ where $z^{m_i}$ for $0 \leq i \leq r$ are regular Laurent monomials on $(\mathbb{C}^*)^s \times \mathbb{C}^t$ and $z^{m_0}$ is invertible on $(\mathbb{C}^*)^s \times \mathbb{C}^t$.  Then for generic values of $a_0, \dots, a_r$, the affine variety $V(a_0, \dots, a_r) \subseteq (\mathbb{C}^*)^s \times \mathbb{C}^t$ defined by $f=0$ is nonsingular. \end{lemma}

Eventually we will be able to state a combinatorial condition for generic members of the family $\mathcal{X}(\Sigma)$ to be smooth.  We begin with the following result, which shows that singularities of generic anticanonical hypersurfaces can only occur at certain zero-dimensional torus orbits of $X(\Sigma)$.

\begin{prop} \label{smoothhypersurface} Let $\Delta$ be a reflexive 4-polytope and $\Sigma$ a $\Delta$-maximal fan.  Then: 
\begin{enumerate}
\item If $s$ is a generic global section of the anticanonical bundle of $X(\Sigma)$, then $s = 0$ defines a hypersurface $Y \subseteq X(\Sigma)$ with at worst isolated singularities.  
\item Any isolated singularities of $Y$ must occur at zero-dimensional toric strata of $X(\Sigma)$ corresponding to cones $\Cone(v_1, \dots, v_4)$ with $v_1, \dots, v_4 \in \partial \Delta$ lattice points not contained in a common face of $\Delta$.
\end{enumerate}
\end{prop}

\begin{proof} To prove part 1, we will prove that $Y \cap X(\Sigma^{[3]})$ is smooth.  This is sufficient to show that $Y$ has isolated singularities since $X(\Sigma)\backslash X(\Sigma^{[3]})$ is zero-dimensional. 

First we verify that the smaller set $Y \cap X(\Sigma^{[2]})$ is smooth.  Let $C=\Cone(v_1, v_2)$ be a two-dimensional cone of $\Sigma$ with $v_1, v_2$ lattice points in $\Delta$.  By Lemma~\ref{lemma:edges}, $v_1$ and $v_2$ are contained in some common face of $\Delta$, so there is a lattice point $m \in \Delta^*$ with $\langle m, v_i \rangle = -1$ for $i = 1, 2$.  Considering the generic global section $s$ restricted to the open affine set $U=$ Spec $\mathbb{C}[\breve{C} \cap M] \subseteq X(\Sigma)$, where $\breve{C}$ is the dual cone to $C$, $c_m z^m$ is an invertible monomial term in $s$.  We also have that $C$ is a smooth cone by Proposition~\ref{prop:3dimsmooth}, implying that $U \cong \mathbb{C}^2 \times (\mathbb{C}^*)^2$.  So by Lemma~\ref{genericsmth}, $Y \cap U$ is smooth.

Now let $C = \Cone(v_1, v_2, v_3)$ be a three-dimensional cone of $\Sigma$ with $v_1, v_2, v_3$ lattice points in $\Delta$ and $U \subseteq X(\Sigma)$ the corresponding open affine.  By Proposition~\ref{prop:3dimsmooth}, $C$ is a smooth cone, and $U \cong \mathbb{C}^3 \times \mathbb{C}^*$.  First assume that $v_1, v_2, v_3$ are contained in a common face.  Then the same argument as in the previous paragraph shows $Y \cap U$ is smooth.

Now suppose that $v_1, v_2, v_3$ are not contained in a common face.  We still have that $U \cong \mathbb{C}^3 \times \mathbb{C}^*$.  In this case, the defining equation of $Y \cap U$ will not have an invertible term as above, but it will have a linear term, which we now show will be enough to imply smoothness.  

Since $C$ is a smooth cone by Proposition~\ref{prop:3dimsmooth}, we can extend $v_1, v_2, v_3$ to a $\mathbb{Z}$-basis of $N$ by adding a vector $u_4 \in N$.  Let the basis of $M$ dual to $\{ v_1, v_2, v_3, u_4 \}$ be $\{v'_1, v'_2, v'_3, u'_4\}$.  Then the dual cone $\breve{C}$ is equal to $\Cone(v'_1, v'_2, v'_3, u'_4, -u'_4)$, and the ring of functions on $U$ is $\mathbb{C}[\breve{C} \cap M] = \mathbb{C}[z^{v'_1}, z^{v'_2},z^{v'_3},z^{u'_4},z^{-u'_4}]$.    

By part 1 of Lemma \ref{uniqueness}, up to a relabeling of $v_1, v_2, v_3$, there exists a lattice point $m \in \Delta^*$ such that $\langle m, v_1 \rangle = \langle m, v_2 \rangle = -1$ and $\langle m, v_3 \rangle = 0$.  On $U$, the lattice point $m$ corresponds to the monomial $z^{v'_3}z^{au'_4}$ where $a = 1+\langle m, u_4 \rangle$.  By part 2 of Lemma \ref{uniqueness}, all other lattice points $\ell \in \Delta^*$ must correspond to monomials $z^{c_1v'_1}z^{c_2v'_2}z^{c_3v'_3}z^{c_4u'_4}$ on $U$, where $c_1, c_2, c_3$ are non-negative integers, and either $c_1+c_2 \geq 1$ or $c_3 \geq 2$.  (This is because $c_i = 1+\langle \ell, v_i \rangle$ for $i = 1, 2, 3$.) 

Let $p \in \mathbb{C}[z^{v'_1}, z^{v'_2},z^{v'_3},z^{u'_4},z^{-u'_4}]$ be a local representative of the section $s$ on $U$.  Let the coefficient of $p$ on the monomial $z^{v'_3}z^{au'_4}$ be $A \in \mathbb{C}$.  Then the partial of $p$ with respect to $z^{v'_3}$ is $$Q=Az^{au'_4}+q(z^{v'_1}, z^{v'_2},z^{v'_3},z^{u'_4},z^{-u'_4})$$ where $q$ is such that $q(0,0,0,z^{u'_4},z^{-u'_4}) = 0$. As $z^{au'_4}$ is invertible on $U$, we can conclude that $Q \neq 0$ when $z^{v'_1} = z^{v'_2} = z^{v'_3} =0$ for generic values of $A$, and therefore the zero locus of $p$ is nonsingular at the set of points in $U$ where $z^{v'_1} = z^{v'_2} = z^{v'_3} =0$.  All other points of $U$ are contained in $X(\Sigma^{[2]})$, where we have already shown $Y$ is smooth.

For part 2, we have established that any singularities of $Y$ must occur in the set $X(\Sigma) \backslash X(\Sigma^{[3]})$, which is the set of zero-dimensional torus orbits in $X(\Sigma)$, corresponding to maximal cones of $\Sigma$. Suppose $C = \Cone(v_1, \dots, v_4)$ is a maximal cone and the $v_i$ are all contained in some common facet of $\Delta$, with dual vertex $m$.  If $U$ is the open affine subset of $X(\Sigma)$ defined by $C$, $Y$ will not contain the unique zero-dimensional torus orbit $t \in U$.  This is because at $t$, the monomial $z^m$ in the defining equation of $Y$ is nonzero while all others vanish.  It follows that $Y \cap U \subseteq X(\Sigma^{[3]})$ is smooth, and any singularities of $Y$ must occur on an open affine corresponding to $\Cone(v_1, \dots , v_4)$ with $v_1, \dots, v_4$ not in a common face of $\Delta$. \end{proof}

Having established Proposition~\ref{smoothhypersurface}, we can now give a sufficient condition for determining when generic members of $\mathcal{X}(\Sigma)$ are smooth, which depends only on the geometry of $\Delta$ and the fan $\Sigma$.  We need only check for smoothness at the remaining zero-dimensional torus orbits specified by part 2 of Proposition~\ref{smoothhypersurface}.  The idea behind the following sufficient condition is that it will ensure the local definining equation of $\mathcal{X}(\Sigma)$ at these points has a lowest term that is linear, and therefore has a nonvanishing partial derivative. 

\begin{prop} \label{smoothness} Let $\Delta$ be a reflexive 4-polytope and $\Sigma$ a $\Delta$-maximal fan.  Then a sufficient condition for generic members of the family $\mathcal{X}(\Sigma)$ to be smooth is as follows: for every maximal cone $\Cone(v_1, v_2, v_3, v_4) \in \Sigma$ with $v_1, \dots, v_4$ lattice points of $\Delta$ that are not contained in a common proper face of $\Delta$, $$v_1+\cdots+v_4 \in 3 \cdot \partial \Delta.$$ \end{prop}

\begin{proof} Assume the condition holds for a cone $C = \Cone(v_1, \dots, v_4)$.  Then $(v_1+\cdots+v_4)/3 \in \partial \Delta$, so it is contained in some maximal proper face $f \subseteq \Delta$ with dual vertex $m \in \Delta^*$.  To have $\langle m, (v_1+\cdots+v_4)/3 \rangle = -1$ we must have $\langle m, v_i \rangle = \langle m, v_j \rangle = \langle m, v_k \rangle = -1$ for exactly three $v_i, v_j, v_k$ and $\langle m, v_\ell \rangle = 0$ for the remaining $v_\ell$.  Assume $i = 1, j = 2, k = 3, \ell = 4$.  

By Proposition~\ref{prop:fourdimsmooth}, we know that $\{ v_1, \dots, v_4 \}$ is a $\mathbb{Z}$-basis of $N$.  If $v'_i$ for $1 \leq i \leq 4$ is the dual basis, then the affine open subset of $X(\Sigma)$ corresponding to $C$ is $U =$ Spec $\mathbb{C}[\breve{C} \cap M] =$ Spec $\mathbb{C}[z^{v'_1}, \dots, z^{v'_4}] \cong \mathbb{C}^4$.  On $U$, $m$ corresponds to the monomial $z^{v'_4}$.  Therefore a local representative of a generic global section of the anticanonical bundle on $U$, $p(z^{v'_1}, \dots, z^{v'_4})$, will have a linear term in $z^{v'_4}$, and no constant term.  It follows immediately that the partial of $p$ with respect to $z^{v'_4}$ will be nonvanishing at the point $z^{v'_1} = \cdots = z^{v'_4} = 0$ (the unique zero-dimensional torus orbit in $U$), implying that generic members of $\mathcal{X}(\Sigma)$ are smooth at this point.
\end{proof}

Based on Proposition~\ref{smoothness}, it is natural to introduce the following condition on a reflexive four-polytope $\Delta$ that will guarantee smoothness of the anticanonical hypersurface families in all $\Delta$-maximal fans.

\begin{definition} \label{maximalcones} Let $\Delta$ be a reflexive four-polytope.  We say that $C$ is a \emph{maximal cone} associated to $\Delta$ if $C = \Cone(v_1, \dots, v_4)$ for linearly independent lattice points $v_1, \dots, v_4 \in \Delta$, and $\Cone(v_1, \dots, v_4) \cap \Delta \cap N = \{0, v_1, \dots, v_4 \}$.  We say that $C$ is a {\it good maximal cone} if either $$v_1+\cdots+v_4 \in 3 \cdot \partial \Delta$$ or $$v_1+\cdots+v_4 \in 4 \cdot \partial \Delta.$$
We say $\Delta$ \emph{has good maximal cones} if all maximal cones associated to $\Delta$ are good.  We say a $\Delta$-maximal fan $\Sigma$ has good maximal cones if all four-dimensional cones of $\Sigma$ are good.
\end{definition} 

\section{Proof of main result via wall crossings}

In this section, we prove our main result, which is that any reflexive four-polytope $\Delta$ has good maximal cones, and thus any $\Delta$-maximal fan $\Sigma$ is associated to an anticanonical hypersurface family $\mathcal{X}(\Sigma)$ whose generic members are smooth.  Before proceeding to the proof, we first introduce necessary ideas from the theory of the secondary fan.

\subsection{The Secondary Fan}

Roughly speaking, the \emph{secondary fan} associated to a list of rational vectors $\nu_1, \dots, \nu_r \in N_\mathbb{R}$ is a way to parametrize certain fans in $N_\mathbb{R}$ with cones generated by $\nu_1, \dots, \nu_r$.  Throughout our discussion, we will assume that $\nu_1, \dots, \nu_r$ span $N_\mathbb{R}$.  We also will always assume that $\nu_1, \dots, \nu_r$ is ``geometric'' as
defined in \cite{CLS}, Section 15.1, meaning that all $\nu_i$ are nonzero and no $\nu_i$ is a positive scalar multiple of another.  We will mostly use the same notation and approach as \cite{CLS}, Chapters 14 and 15.  Other sources on the secondary fan (also known as the GKZ decompsition for its discoverers Gel'fand, Kapranov, and Zelevinsky) include \cite{coxkatz, odapark, GKZbook}.

The secondary fan consists of cones associated to certain pairs $(\Sigma, I_\emptyset)$, where $I_\emptyset \subseteq \{ 1, \dots, r \}$, and $\Sigma$ is a ``generalized fan'' in $N_\mathbb{R}$.  A generalized fan has the same definition as a fan (a finite set of rational convex polyhedral cones in $N_\mathbb{R}$ such that if $C_1, C_2 \in \Sigma$, $C_1 \cap C_2 \in \Sigma$, and if $C'$ is a face of $C \in \Sigma$, then $C' \in \Sigma$), but the cones are
not required to be pointed; in other words, they may contain a nontrivial linear subspace of $N_\mathbb{R}$.  A generalized fan which is a fan in the usual sense is called nondegenerate. Otherwise, it is said to be degenerate.

The \emph{support} $|\Sigma|$ of a generalized fan $\Sigma$ is the set of points in $N_\mathbb{R}$ contained in some cone of $\Sigma$.  A \emph{support function} for $\Sigma$ is a continuous piecewise linear function $\varphi : |\Sigma| \rightarrow \mathbb{R}$, such that the restriction of $\varphi$ to each cone of $\Sigma$ is linear.  A support function $\varphi: |\Sigma| \rightarrow \mathbb{R}$ is said to be \emph{strictly convex} if it is convex and its maximal domains of linearity are exactly the maximal cones of $\Sigma$.

The following definition is the same as \cite{CLS}, Definition 14.4.2.

\begin{definition} \label{gkzdef} Suppose we have a spanning geometric set of rational vectors $\nu_1, \dots \nu_r \in N_\mathbb{R}$. Consider pairs $(\Sigma, I_\emptyset)$ satisfying the following:
\begin{enumerate}
\item $\Sigma$ is a generalized fan in $N_\mathbb{R}$ and $I_\emptyset \subseteq \{1, \dots, r \}$.
\item The support of $\Sigma$, $|\Sigma|$, equals $\Cone(\nu_1, \dots, \nu_r)$.
\item There exists a strictly convex support function for $\Sigma$.
\item For each $\sigma \in \Sigma$, we have that $\sigma = \Cone(\nu_i \ | \ \nu_i \in \sigma, i \not \in I_\emptyset).$
\end{enumerate}
Then the \emph{GKZ cone} $\widetilde{\Gamma}_{\Sigma,I_\emptyset} \subseteq \mathbb{R}^r$ is defined as 
\begin{gather*}
\widetilde{\Gamma}_{\Sigma,I_\emptyset} = \{ (a_1, \dots, a_r) \in \mathbb{R}^r \ | \  \hbox{there is a convex piecewise linear support function} \\ \varphi: |\Sigma| \rightarrow \mathbb{R}
\ \hbox{such that} \ \varphi(\nu_i) = -a_i \ \hbox{for} \ i \not \in I_\emptyset \ \hbox{and} \ \varphi(\nu_i) \geq -a_i \ \hbox{for} \ i \in I_\emptyset \}.
\end{gather*}
\end{definition}

The results of \cite{CLS}, Section 14.4, show that if we consider all possible pairs $(\Sigma, I_\emptyset)$ allowed by Definition \ref{gkzdef}, then the GKZ cones $\widetilde{\Gamma}_{\Sigma,I_\emptyset}$ fit together to form a generalized fan in $\mathbb{R}^r$, called $\widetilde{\Sigma}_{\rm{GKZ}}$.  The minimal cone of $\widetilde{\Sigma}_{\rm{GKZ}}$ (the intersection of all its cones) is the linear subspace $L \subseteq \mathbb{R}^r$ defined by
$$L = \{ (\ell(\nu_1), \dots, \ell(\nu_r)) \ | \ \ell : N_\mathbb{R} \rightarrow \mathbb{R} \hbox{ is a linear map} \}.$$
The notation used by \cite{CLS} for the quotient space $\mathbb{R}^r/L$ is $\widehat{G}_\mathbb{R}$.  The image of each GKZ cone in $\widehat{G}_\mathbb{R}$ will be a pointed cone.  Thus, we get a nondegenerate fan $\Sigma_{\rm{GKZ}}$ in $\widehat{G}_\mathbb{R}$.  
By definition, $\Sigma_{\rm{GKZ}}$ is the secondary fan associated to the vectors $\nu_1, \dots, \nu_r$.  We use the notation $\Gamma_{\Sigma, I_\emptyset}$ for the cone in $\Sigma_{\rm{GKZ}}$ which is the image of $\widetilde{\Gamma}_{\Sigma,I_\emptyset}$.

We have the following important results about the structure of $\Sigma_{\rm{GKZ}}$ (see \cite{CLS}, Theorem 14.4.7 and Proposition 14.4.9).
\begin{prop} \label{gkzproperties}
Let $\nu_1, \dots, \nu_r$ be a spanning geometric set of rational vectors in $N_\mathbb{R}$ and let $\Sigma_{\rm{GKZ}}$ be the corresponding secondary fan in $\widehat{G}_\mathbb{R}$.  Then:
\begin{enumerate}
\item The support of $\Sigma_{\rm{GKZ}}$ is convex and full-dimensional in $\widehat{G}_\mathbb{R}$.
\item Every face of a cone $\Gamma_{\Sigma, I_\emptyset}$ is a cone of the form $\Gamma_{\Sigma', I'_\emptyset}$ for some pair $(\Sigma', I'_\emptyset)$.
\item $\Gamma_{\Sigma', I'_\emptyset}$ is a face of $\Gamma_{\Sigma, I_\emptyset}$ if and only if $\Sigma$ refines $\Sigma'$ and $I'_\emptyset \subseteq I_\emptyset$.
\item $\Gamma_{\Sigma, I_\emptyset}$ is a maximal cone of $\Sigma_{\rm{GKZ}}$ if and only if $\Sigma$ is simplicial and the rays of $\Sigma$ are exactly the rays over $\nu_i$ for $i \not \in 
I_\emptyset$.
\end{enumerate}
\end{prop}

From here on, $\nu_1, \dots, \nu_r$ will always be the elements of $\partial \Delta \cap N$, where $\Delta \subseteq N_\mathbb{R}$ is some fixed reflexive polytope.  This implies that $\Cone(\nu_1, \dots, \nu_r) = N_\mathbb{R}$.  Therefore, all generalized fans $\Sigma$ defining a cone $\Gamma_{\Sigma, I_\emptyset} \in \Sigma_{\rm{GKZ}}$ will have support equal to $N_\mathbb{R}$.  We also have that $\nu_1, \dots, \nu_r$ is ``primitive geometric'', meaning that $\nu_1, \dots, \nu_r$ is geometric, along with the extra condition that each $\nu_i$ is a primitive element of $N$.  (An element $n \in N$ is primitive if it is not a positive integer multiple of some other element of $N$.)

The \emph{moving cone} of the secondary fan, $\rm{Mov}_{\rm{GKZ}}$, is defined as the union of all cones $\Gamma_{\Sigma, I_\emptyset} \in \Sigma_{\rm{GKZ}}$ where $I_\emptyset$ is the empty set.  Since $\nu_1, \dots, \nu_r$ is geometric, Propositions 15.1.4 and 15.1.5 of \cite{CLS} say that $\rm{Mov}_{\rm{GKZ}}$ is a full-dimensional convex polyhedral cone in $\widehat{G}_\mathbb{R}$.

In our case, it follows from the facts of Proposition \ref{gkzproperties} that maximal cones in the moving cone of the secondary fan correspond to complete projective simplicial fans whose rays are equal to the rays over the elements of $\partial \Delta \cap N$, in other words, projective $\Delta$-maximal fans.

\subsection{Wall Crossings and Circuits}

To prove our main result we will use the theory of \emph{wall crossings} in the secondary fan, as explained in \cite{CLS}, Section 15.3.  Suppose we have distinct maximal cones $\Gamma_{\Sigma, \emptyset}$, $\Gamma_{\Sigma',\emptyset}$ in the moving cone that share a common codimension-1 face $\Gamma_{\Sigma_0, \emptyset}$. Then $\Sigma$ and $\Sigma'$ are said to be related by a wall crossing, and $\Gamma_{\Sigma_0, \emptyset}$ is called a ``flipping wall''.
$\Sigma$ and $\Sigma'$ are both subdivisions of $\Sigma_0$, which is non-simplicial.  In our case, $\Sigma$ and $\Sigma'$ are both projective $\Delta$-maximal fans.  It also follows from \cite{CLS}, Proposition 14.4.12, part b, that $\Sigma_0$ is a nondegenerate fan.

It turns out that a wall crossing is described by the data of an \emph{oriented circuit} in $N_\mathbb{R}$, which we now define.

Let $n = \hbox{dim} \ N_\mathbb{R}$.  Suppose we have a set of $n+1$ nonzero vectors $W = \{ w_1, \dots, w_{n+1} \} \subseteq N_\mathbb{R}$ which span $N_\mathbb{R}$.  For our purposes we can assume that all $w_1, \dots, w_{n+1}$ are elements of $N$.  There is a linear dependence equation 
$$\sum_{i=1}^{n+1} b_i w_i = 0$$
which is unique up to rescaling, with the $b_i$ integers not all equal to zero.  Set $W_+$ to be the set of $w_i$ with $b_i >0$, $W_0$ the set of $w_i$ with $b_i = 0$, and $W_-$ the set of $W_i$ with $b_i <0$.  

The set of vectors $W_+ \cup W_-$ has the property that it is linearly dependent, but any proper subset is linearly independent.  A set of vectors with this property is called a circuit.  Rescaling the dependence equation will at most swap $W_+$ and $W_-$.  A fixed choice of $W_-$ and $W_+$ makes $W$ an oriented circuit.

Define the sets of simplicial cones $\Sigma_-$ and $\Sigma_+$ as 
\begin{align*}
\Sigma_- &= \{ \sigma \ | \ \sigma \hbox{ is a face of } \Cone(w_1, \dots, \widehat{w}_i, \dots, w_{n+1}) \hbox{ for some } w_i \in W_+  \} \\
\Sigma_+ &= \{ \sigma \ | \ \sigma \hbox{ is a face of } \Cone(w_1, \dots, \widehat{w}_i, \dots, w_{n+1}) \hbox{ for some } w_i \in W_-  \},
\end{align*} 
where $\widehat{w}_i$ indicates that $w_i$ is omitted.  Then we have the following result (Lemma 15.3.10, \cite{CLS}):

\begin{lemma} \label{circuitcones} Suppose that $C = \Cone(w_1, \dots, w_{n+1})$ is pointed.  Then $W_-$ and $W_+$ are both nonempty, and $\Sigma_-$ and $\Sigma_+$ are simplicial fans with support $C$. \end{lemma}

We can now describe how the main result on wall crossings in the secondary fan (Theorem 15.3.13, \cite{CLS}) applies to our situation of interest.  Let $\nu_1, \dots, \nu_r$ be the elements of $\partial \Delta \cap N$ and let $\Sigma_{\rm{GKZ}}$ be the associated secondary fan.  For a set $S \subseteq \{1, \dots, r \}$, define $\sigma_S = \Cone(\nu_i \ | \ i \in S)$.  Also, given any fan $\Sigma$ in $N_\mathbb{R}$ and a cone $C \subseteq N_\mathbb{R}$, define $\Sigma|_C$ as  $\{ \sigma \ | \ \sigma \in \Sigma, \sigma \subseteq C \}$, a subfan of $\Sigma$.

\begin{thm} \label{wallcrossingsthm}  Let $\Sigma$ and $\Sigma'$ be projective $\Delta$-maximal fans such that $\Gamma_{\Sigma, \emptyset}$ and $\Gamma_{\Sigma', \emptyset}$ share a common codimension-1 face $\Gamma_{\Sigma_0, \emptyset}$ in $\Sigma_{\rm{GKZ}}$.  Then $\Sigma$ and $\Sigma'$ are refinements of the nondegenerate complete fan $\Sigma_0$ which is not simplicial.  The set of rays of $\Sigma_0$ is equal to the set of rays over $\nu_1, \dots, \nu_r$.  There are disjoint sets $J_-, J_+ \subseteq \{1, \dots, r \}$ such that we have an equation
$$\sum_{i \in J_+} b_i \nu_i + \sum_{i \in J_-} b_i \nu_i = 0$$
where the $b_i$ are integers such that $b_i < 0$ if $i \in J_-$ and $b_i > 0$ if $i \in J_+$.  

We have the following facts: 
\begin{enumerate}
\item $\sigma_{J_- \cup J_+} \in \Sigma_0$, and either $\sigma_{J_+} \in \Sigma$ and $\sigma_{J_-} \in \Sigma'$, or $\sigma_{J_+} \in \Sigma'$ and $\sigma_{J_-} \in \Sigma$.  (Up to relabeling of $\Sigma$ and $\Sigma'$, we can assume the former case.)
\item $J_-$ and $J_+$ both have cardinality at least two.  If $W_+ = \{ \nu_i \ | \ i \in J_+ \}$
$W_- = \{ \nu_i \ | \ i \in J_- \}$, then the set of vectors $W_+ \cup W_-$ forms an oriented circuit.
\item If $\sigma$ is a maximal non-simplicial cone in $\Sigma_0$, then there is a $J_0 \subseteq \{1, \dots, r \}$ such that $J = J_+ \cup J_0 \cup J_-$ is a disjoint union and $\sigma = \sigma_J$.  If we let $W_0 = \{ \nu_i \ | \ i \in J_0 \}$,  then $W = W_+ \cup W_0 \cup W_-$ has cardinality $1+{\rm{dim}} \ N_\mathbb{R}$, and the subfans $\Sigma|_{\sigma}$ and $\Sigma'|_{\sigma}$ are respectively equal to the fans $\Sigma_+$ and $\Sigma_-$ from Lemma \ref{circuitcones}.
\end{enumerate}
  \end{thm}

\subsection{Proof of main result} We would like to show that any reflexive four-polytope $\Delta$ has good maximal cones.  The strategy will be as follows.  Any MPCP resolution of $\Sigma(\Delta)$ trivially has good maximal cones since its maximal cones are of the form $\Cone(v_1, \dots, v_4)$ with $v_1, \dots, v_4$ in a common facet of $\Delta$.  At least one MPCP resolution of $\Sigma(\Delta)$ always exists by the results of \cite{Ba}, and represents a maximal cone in the moving cone of $\Sigma_{\rm{GKZ}}$.  Since the moving cone is convex of full dimension, we can move from any maximal cone in $\rm{Mov}_{\rm{GKZ}}$ to another by a finite number of wall crossings.  Thus it suffices to show that if $\Sigma$ and $\Sigma'$ are projective $\Delta$-maximal fans related by a wall crossing, and one has good maximal cones, then the other does as well.

Once we know that all projective $\Delta$-maximal fans have good maximal cones, we will prove that $\Delta$ has good maximal cones by showing that any maximal cone of $\Delta$ must be a cone in some projective $\Delta$-maximal fan.

Thus, let $\Delta$ be a reflexive four-polytope and $\Sigma$ and $\Sigma'$ be two projective $\Delta$-maximal fans related by a wall crossing, as in Theorem \ref{wallcrossingsthm}.  The proof will be divided into two cases, depending on the cardinality of $J_- \cup J_+$.  

First suppose that $J_- \cup J_+$ has four elements.  Then it turns out that the lattice points in $W_-$ and $W_+$ will be the opposite vertices of a square.  This is the classical picture of an Atiyah flop.  See Figure \ref{flop}.

\begin{figure}[h] 
  \begin{center} \scalebox{0.6}{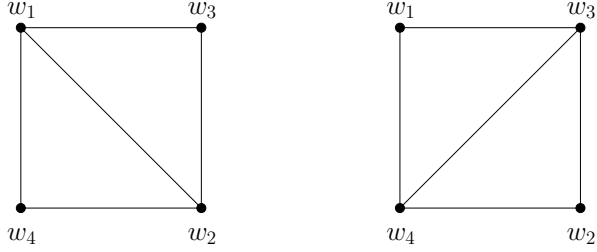} \end{center}
   \caption{\label{flop} Cross sections of the fan $\Sigma$ and $\Sigma'$} 
\end{figure}

\begin{lemma} Let $e_1, \dots, e_4$ be the standard basis of $N \cong \mathbb{Z}^4$.  Up to a $\mathbb{Z}$-linear isomorphism, $W_+ = \{ e_1, e_2 \}$, and $W_- = \{ e_3, -e_3+e_1+e_2 \}$. \end{lemma}

\begin{proof} Let $W_+ = \{ w_1, w_2 \}$ and $W_- = \{ w_3, w_4 \}$.  Up to relabeling of $\Sigma$ and $\Sigma'$, it follows from Lemma \ref{circuitcones} and Theorem \ref{wallcrossingsthm} that 
\begin{align*}
\Cone(w_1, w_2, w_3), \Cone(w_1, w_2, w_4) &\in \Sigma \\ 
\Cone(w_1, w_3, w_4), \Cone(w_2, w_3, w_4) &\in \Sigma'.
\end{align*}  
Each of these four cones is smooth by Proposition \ref{prop:3dimsmooth}.  Since $\Cone(w_1, w_2, w_3)$ is smooth, there is a $\mathbb{Z}$-linear isomorphism sending $w_i$ to $e_i$ for $i = 1, 2, 3$.  By definition of $W_+$ and $W_-$, we have that $b_1 w_1 + b_2 w_2 = b_3 w_3 + b_4 w_4$ for positive integers $b_i$.  It follows from smoothness of the other three cones that $w_4 = -e_3+e_1+e_2$. \end{proof}

\begin{prop} \label{fourelements} If $J_- \cup J_+$ has four elements and either $\Sigma$ or $\Sigma'$ has good maximal cones, then both do.
\end{prop}

\begin{proof} Let $W_+ = \{ w_1, w_2 \}$ and $W_- = \{ w_3, w_4 \}$.  According to Theorem \ref{wallcrossingsthm}, part 3, any cones of $\Sigma$ not contained in $\Sigma'$ must be contained in a cone of $\Sigma_0$
of the form $\Cone(w_1, w_2, w_3, w_4, w_5)$ for some other lattice point $w_5$ of Delta.  In the notation of Theorem \ref{wallcrossingsthm}, $\Cone(w_1, w_2, w_3, w_4, w_5) = \sigma_J$.  (There must be only one generator in addition to $w_1, \dots, w_4$ because the cardinality of $J$ must equal $1+{\rm{dim}} \ N_\mathbb{R} = 5$.)  We must
have that either $C_1 = \Cone(w_1, w_2, w_3, w_5)$ and $C_2 = \Cone(w_1, w_2, w_4, w_5)$ are in $\Sigma$, or $C_3 = \Cone(w_1, w_3, w_4, w_5)$ and $C_4 = \Cone(w_2, w_3, w_4, w_5)$ are in $\Sigma$.

Assume without loss that $C_1$ and $C_2$ are in $\Sigma$.  By assumption, $\Sigma$ has good maximal cones.  We must show that $\Sigma'$ also has good maximal cones, which means showing that
the cones $C_3, C_4 \in \Sigma'$ are good maximal cones.  There are several cases to consider.

Case 1: Either $w_1+w_2+w_3+w_5 \in 4 \cdot \partial \Delta$ or $w_1+w_2+w_4+w_5 \in 4 \cdot \partial \Delta.$  Assume without loss that we 
are in the first case and $w_1+w_2+w_3+w_5 \in 4 \cdot \partial \Delta$.  By Lemma \ref{lemma:commonface}, there must be some facet $F \subseteq \Delta$ such that
$\langle u_F, w_i \rangle = 1$ for i = 1, 2, 3, and 5.  It follows by linearity that $\langle u_F, w_4 \rangle = 1$, so $w_i$ for $1 \leq i \leq 5$ are all contained in $F$.  Thus by Lemma \ref{lemma:commonface}, $w_1+w_3+w_4+w_5$ and $w_2+w_3+w_4+w_5$ are in $4 \cdot \partial \Delta$.

Case 2: Both $w_1+w_2+w_3+w_5$ and $w_1+w_2+w_4+w_5$ are in $3 \cdot \partial \Delta$.  Then there are faces $F, G \subseteq \Delta$ such that $u_F$ 
evaluates to 1 on exactly three of $w_1, w_2, w_3, w_5$ and 0 on the remaining lattice point, and $u_G$ evaluates to 1 on exactly three of 
$w_1, w_2, w_4, w_5$ and 0 on the remaining lattice point.  We divide this into three subcases, 2A, 2B, and 2C.

Case 2A: Either $\langle u_F, w_5 \rangle = 0$ or $\langle u_G, w_5 \rangle = 0$.  Then it follows by linearity that either $\langle u_F, w_i \rangle = 1$ for all $1 \leq i \leq 4$ or 
$\langle u_G, w_i \rangle = 1$ for all $1 \leq i \leq 4$.  In either case we have that one of $u_F$ or $u_G$ evaluates to 3 on $w_1+w_3+w_4+w_5$ and $w_2+w_3+w_4+w_5$,
so $C_3$ and $C_4$ are good maximal cones.

Case 2B: $u_F$ evaluates to identically zero on $E_1$ and $u_G$ evaluates to identically zero on $E_2$, where $E_1$ and $E_2$ are opposite
edges of the square $\Conv(w_1, w_2, w_3, w_4)$.  Without loss, assume that $E_1 = \Conv(w_1, w_3)$ and $E_2 = \Conv(w_2, w_4)$.  Then 
$\langle u_F, w_2 \rangle = \langle u_F, w_4 \rangle = \langle u_F, w_5 \rangle = 1$ and $\langle u_F, w_3 \rangle = 0$ implies that $C_4$ is a good maximal cone, while $\langle u_G, w_1 \rangle = \langle u_G, w_3 \rangle = 
\langle u_G, w_5 \rangle = 1$ and $\langle u_G, w_4 \rangle = 0$ implies $C_4$ is a good maximal cone.

Case 2C: $u_F$ evaluates to identically zero on $E_3$ and $u_G$ evaluates to identically zero on $E_4$, where $E_3$ and $E_4$ are adjacent
edges of the square $\Conv(w_1, w_2, w_3, w_4)$.  Without loss, assume that $E_3 = \Conv(w_1, w_3)$ and $E_4 = \Conv(w_1, w_4)$.  Then we have that $\langle u_F, w_2 \rangle = \langle u_F, w_4 \rangle = \langle u_F, w_5 \rangle = 1$ and $\langle u_G, w_2 \rangle = \langle u_F, w_3 \rangle = \langle u_F, w_5 \rangle = 1$.

First assume that $w_1, w_2$ and $w_5$ are in a common facet of $\Delta$.  Then there is some facet $H \subseteq \Delta$ such that
$\langle u_H, w_1 \rangle = \langle u_H, w_2 \rangle = \langle u_H, w_5 \rangle = 1$.  By integrality and the fact that $\langle u_H, p \rangle$ must be at most 1 for any $p \in \Delta$, it follows that 
$\langle u_H, w_3 \rangle = \langle u_H, w_4 \rangle = 1$ as well.  But then we would be in Case 1.

Now assume $w_1, w_2, w_5$ are not in a common facet.  We claim this is impossible.  Notice that $-u_F$ and $-u_G$ are distinct lattice points of $\Delta^*$.  Also, $-u_F$ and $-u_G$ both evaluate to $-1$ on $w_2$ and $w_5$, and evaluate to 0 on $w_1$.  We have that $\Cone(w_1, w_2, w_5)$ is a cone in the $\Delta$-maximal fan $\Sigma$.  By part 2 of Lemma \ref{uniqueness}, this situation is impossible.
\end{proof}

\begin{prop} \label{fiveelements} If $J_- \cup J_+$ has five elements and either $\Sigma$ or $\Sigma'$ has good maximal cones, then both do.
\end{prop}

\begin{figure}[h] 
  \begin{center} \scalebox{0.6}{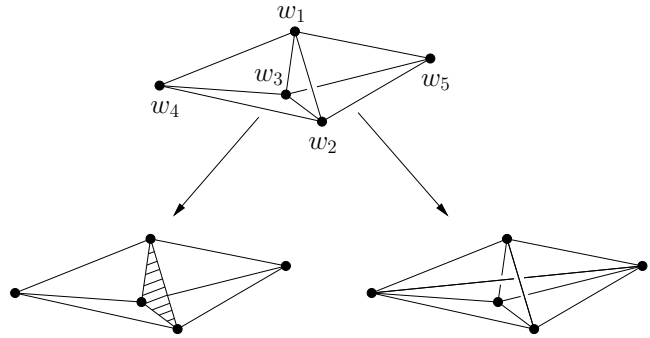} \end{center}
   \caption{\label{wallcrossing} Cross sections of the fan $\Sigma_0$ (top), $\Sigma$ (lower left), and $\Sigma'$ (lower right)} 
\end{figure}

\begin{proof} We may assume up to interchanging $J_+$ and $J_-$ that $J_+$ has three elements and $J_-$ has two elements.  Then $W_+ = \{ w_1, w_2, w_3 \}$ and $W_-= \{ w_4, w_5 \}$. 

$\Cone(w_1, w_2, w_3)$ must be a cone of either $\Sigma$ or $\Sigma'$; in the following proof we can assume without loss that it is a 
cone of $\Sigma$.  Then the maximal cones of $\Sigma$ not contained in $\Sigma'$ must be
exactly $\Cone(w_1, w_2, w_3, w_4)$ and $\Cone(w_1, w_2, w_3, w_5)$.  Likewise, the maximal cones of $\Sigma'$ not contained in $\Sigma$ must be
exactly 
$$
\Cone(w_1, w_2, w_4, w_5), \  
\Cone(w_1, w_3, w_4, w_5), \
\Cone(w_2, w_3, w_4, w_5).
$$
See Figure 2 for cross sections of the relevant cones in $\Sigma_0$, $\Sigma$, and $\Sigma'$.  (Note this figure may be somewhat misleading, because the five vertices $w_1, \dots, w_5$ will not always be contained in a common affine hyperplane.)

First suppose that $w_1, w_2, w_3$ are contained in some common facet $F \subseteq \Delta$.  Let $u_F$ be the facet normal to $F$.  Applying $u_F$ to
the linear dependence equation $b_1w_1+b_2w_2+b_3w_3 = b_4w_4+b_5w_5$, we get $b_1+b_2+b_3 = b_4 \langle u_F,w_4 \rangle+b_5 \langle u_F,w_5 \rangle$.  Since $\Cone(w_4, w_5)$ is a cone in $\Sigma'$,
we have by Lemma \ref{lemma:edges} that $w_4$ and $w_5$ are contained in some common facet of $\Delta$, say $G$.  Applying $u_G$ to the same
equation gives $b_1 \langle u_G, w_1 \rangle + b_2 \langle u_G,w_2 \rangle + b_3 \langle u_G,w_3 \rangle=b_4+b_5$.  Since $\langle u_F, w_i \rangle, \langle u_G, w_i \rangle \leq 1$ for all $1 \leq i \leq 5$, 
and all the $b_i$ are positive, we get that $b_1+b_2+b_3 \leq b_4+b_5$ and $b_1+b_2+b_3 \geq b_4+b_5$, so $b_1+b_2+b_3 = b_4+b_5$.  This also implies that
$\langle u_F, w_4 \rangle = \langle u_F, w_5 \rangle = 1$, so $w_1, \dots, w_5$ are all in a common facet, and $\Sigma$ and $\Sigma'$ must both have good maximal cones.

Now assume that $w_1, w_2, w_3$ are not contained in a common facet.  By part 1 of Lemma \ref{uniqueness}, there is a vertex $m$ of $\Delta^*$, with corresponding facet normal $u_F = -m$, such that (up to relabeling of $w_1$, $w_2$, $w_3$) $\langle u_F, w_1 \rangle = \langle u_F, w_2 \rangle = 1$, and $\langle u_F, w_3  \rangle = 0$.

Since $\Cone(w_1, w_2, w_3, w_4)$ and $\Cone(w_1, w_2, w_3, w_5)$ are cones in
the maximal fan $\Sigma$, $\{ w_1, w_2, w_3, w_4 \}$ and $\{ w_1, w_2, w_3, w_5 \}$ must be lattice bases by Proposition \ref{prop:3dimsmooth}.  Up to a $\mathbb{Z}$-linear isomorphism we may
assume that $w_i = e_i$ for $1 \leq i \leq 4$, where $e_i$ is the standard basis of $N \cong \mathbb{Z}^4$.  Then if $\{w_1, w_2, w_3, w_5 \}$ is a lattice basis, we must have that $w_5 = a w_1 +b w_2 + c w_3- w_4$
for some integers $a, b, c$.  This implies that $w_4+w_5 = a w_1 +b w_2 +c w_3$.  By definition of a circuit, there can be only one equation of linear dependence between $w_1, \dots, w_5$
up to rescaling.  Since $b_4 w_4 +b_5 w_5 = b_1 w_1 +b_2 w_2 + b_3 w_3$ for positive $b_i$, we must have that $a, b, c$ are positive integers.  Applying $u_F$ on both sides,
we get that $\langle u_F, w_4 \rangle + \langle u_F, w_5 \rangle = a+b$.  Since the LHS is at most two and the RHS is at least two, we get that $a = b = \langle u_F, w_4 \rangle =
\langle u_F, w_5 \rangle = 1$.  This implies that $\langle u_F, w_1+w_2+w_3+w_4+w_5 \rangle = 4$, and $w_1+w_2+w_3+w_4+w_5 \in 4 \cdot \partial \Delta$.  Then for any distinct $1 \leq i, j, k, \ell \leq 5$, we get that $w_i + w_j + w_k + w_\ell \in 3 \cdot \partial \Delta$ or $4 \cdot \partial \Delta$ (since if $w_i + w_j + w_k + w_\ell \in 2 \Delta$,
adding the remaining $w_m$ for $1 \leq m \leq 5$ would give a lattice point in $3 \Delta$).  Thus $\Sigma$ and $\Sigma'$ both have good maximal
cones.
\end{proof}

Propositions \ref{fourelements} and \ref{fiveelements} establish that any projective $\Delta$-maximal fan has good maximal cones.  By the following lemma, this is enough to imply that $\Delta$ has good maximal cones.

\begin{lemma} Let $\Delta$ be a reflexive four-polytope and let $C$ be a maximal cone of $\Delta$.  Then there exists a projective $\Delta$-maximal fan $\Sigma$ with $C \in \Sigma$. \end{lemma}

\begin{proof}  Let $C = \Cone(v_1, v_2, v_3, v_4)$ with $v_1, \dots, v_4$ lattice points in $\Delta$.  Let the vertices of $\Delta$ be $w_1, \dots, w_r$.  Consider the polytope $$P = \Conv(av_1, av_2, av_3, av_4, w_1, \dots, w_r)$$ where $a$ is some sufficiently large positive number.  Then $P$ will contain the origin in its interior and have $\Conv(av_1, av_2, av_3, av_4)$ as a face.  The fan of cones over proper faces of $P$, $\Sigma(P)$, is projective and will have $C \in \Sigma(P)$.  

Let $\Sigma_{\rm{GKZ}}$ be the secondary fan associated to the elements $\nu_1, \dots, \nu_r$ of $\partial \Delta \cap N$. Then $\Gamma_{\Sigma(P),\emptyset}$ is cone in the moving cone of $\Sigma_{\rm{GKZ}}$.  Let $\Sigma'$ be such that $\Gamma_{\Sigma', \emptyset}$ is a maximal cone of $\Sigma_{\rm{GKZ}}$ in the moving cone with $\Gamma_{\Sigma(P),\emptyset} \subseteq \Gamma_{\Sigma', \emptyset}$.  Then $\Sigma'$ is a projective $\Delta$-maximal fan which is a refinement of $\Sigma(P)$, in particular, $C \in \Sigma'$. \end{proof}

We can now state the following main result:

\begin{thm} \label{goodness} Any reflexive 4-polytope $\Delta$ has good maximal cones.  If $\Sigma$ is any (possibly non-projective) $\Delta$-maximal fan, then generic members of the anticanonical hypersurface family $\mathcal{X}(\Sigma)$ are smooth. \end{thm}

\begin{rmk} {\rm If we view the secondary fan as parametrizing different compactifications of birational hypersurface families, then our main result can be understood as extending smooth CY compactifications to the entire moving cone, rather than just the set of cones corresponding to MPCP resolutions (which in general will form a smaller cone that is strictly contained in the moving cone).

As a related note, the secondary fan is an object of interest in mirror symmetry because it is related to the K\"ahler moduli space of a toric hypersurface family, as explained in \cite{coxkatz}, Section 6.2.  From this point of view, a wall crossing in the secondary fan corresponds to a degeneration of the K\"ahler metric resulting in a birational contraction of a CY family to a singular family, followed by resolution to a possibly different CY family.}
\end{rmk}

\section{Example of CY-threefold from $\Delta$-maximal fan}

In this section we will look at the main example from \cite{fred3} in the context of this paper.  Fix the reflexive polytope $\Delta$ in $N \cong \mathbb{Z}^4$ as 
\begin{gather*}
\Delta = \Conv((1,0,0,0),(0,1,0,0),(0,0,1,0),(0,0,0,1), \\
(1,1,1,1),(-1,-1,-1,-1),(0,0,0,-1)). 
\end{gather*}
This is a smooth Fano 4-polytope with twelve facets, and we will compute all the possible $\Delta$-maximal fans for this polytope.  Because of unimodularity, the fan $\Sigma(\Delta)$ of cones over faces of $\Delta$ is a $\Delta$-maximal fan, associated to the standard family of smooth hypersurfaces $\mathcal{X}(\Delta)$.

A $\Delta$-maximal fan different from $\Sigma(\Delta)$ must contain a cone $\Cone(v_1, \dots, v_4)$ such that $v_1, \dots, v_4$ are vertices of $\Delta$ forming a basis of $N$, and $\Conv(v_1, \dots, v_4)$ is not a face of $\Delta$.  There are exactly four such cones:
\begin{align*}
C_1 &= \Cone((1,0,0,0),(0,1,0,0),(0,0,1,0),(0,0,0,-1)) \\
C_2 &= \Cone((1,0,0,0),(0,1,0,0),(0,0,1,0),(1,1,1,1)) \\
C_3 &= \Cone((1,0,0,0),(0,1,0,0),(0,0,1,0),(-1,-1,-1,-1)) \\
C_4 &= \Cone((1,0,0,0),(0,1,0,0),(0,0,1,0),(0,0,0,1)).
\end{align*}
The cones $C_3$ and $C_4$ respectively contain the vertices $(0,0,0,-1)$ and $(1,1,1,1)$ of $\Delta$, so they cannot be cones in a $\Delta$-maximal fan.  However, $C_1$ and $C_2$ contain no additional vertices of $\Delta$.  

From this information, it can be deduced that there is exactly one $\Delta$-maximal fan other than $\Sigma(\Delta)$.  Using the same notation as \cite{fred3}, we will call this fan $\Sigma_{bl}$ (for reasons that will be explained shortly) and the associated Calabi-Yau family $\mathcal{X}_{bl}$.  $\Sigma_{bl}$ can be obtained from $\Sigma(\Delta)$ by removing the maximal cones 
\begin{align*}
C_5 &= \Cone((1,0,0,0),(0,1,0,0),(1,1,1,1),(0,0,0,-1)) \\
C_6 &= \Cone((1,0,0,0),(0,0,1,0),(1,1,1,1),(0,0,0,-1)) \\
C_7 &= \Cone((0,1,0,0),(0,0,1,0),(1,1,1,1),(0,0,0,-1)) 
\end{align*}
and replacing them with $C_1$ and $C_2$.  $\Sigma_{bl}$ can also be obtained from the fan of cones over the polytope $$\Conv((1,0,0,0),(0,1,0,0),(0,0,1,0),(0,0,0,1),(-1,-1,-1,-1)),$$ which is the fan for $\mathbb{P}^4$, by blowing up at the zero-dimensional toric strata associated to the two maximal cones 
\begin{gather*}
\Cone((1,0,0,0),(0,1,0,0),(0,0,1,0),(0,0,0,1)) \\
\Cone((1,0,0,0),(0,1,0,0),(0,0,1,0),(-1,-1,-1,-1)).
\end{gather*}
Because of this, the toric variety associated to $\Sigma_{bl}$ is isomorphic to $\mathbb{P}^4$ blown up at two points. 
Note that generic members of $\mathcal{X}_{bl}$ are guaranteed to be smooth by Theorem~\ref{goodness}.

If we intersect the cones $C_1, C_2 \in \Sigma_{bl}$ or $C_5, C_6, C_7 \in \Sigma(\Delta)$ with a transverse hyperplane, we can see the difference between $\Sigma_{bl}$ and $\Sigma(\Delta)$ in terms of three-dimensional polytopes.  In Figure~\ref{crosssections} we see that the cross-section for $\Sigma_{bl}$ consists of two tetrahedra sharing a common facet, while the cross-section for $\Sigma(\Delta)$ consists of three tetrahedra sharing a common edge.  This is the same type of wall crossing that occurred in the proof of Proposition \ref{fourelements} (see  Figure~\ref{wallcrossing}).  Also compare \cite{CLS}, Example 15.3.9.  

\begin{figure}[h] 
  \begin{center} \includegraphics[scale=0.6]{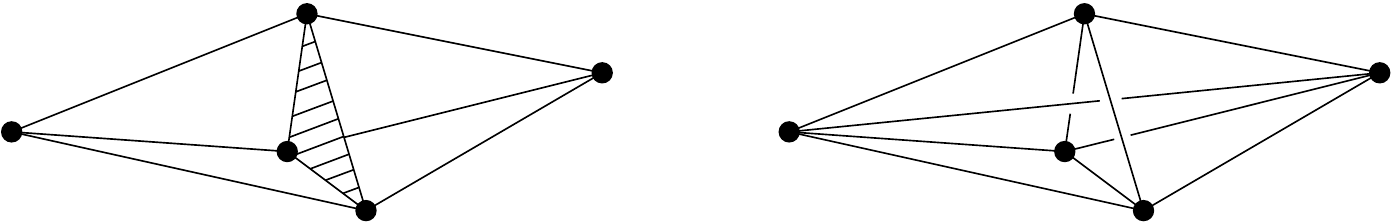} \end{center}
   \caption{\label{crosssections} Cross sections of the fan $\Sigma_{bl}$ (on the left) and $\Sigma(\Delta)$ (on the right)} 
\end{figure}

Locally analyzing the hypersurface families $\mathcal{X}(\Delta)$ and $\mathcal{X}_{bl}$ on these cones shows that they are related by a flop.  The flop contracts a single $\mathbb{P}^1$ in each member of $\mathcal{X}(\Delta)$ or $\mathcal{X}_{bl}$ to an ordinary double point, then resolves the ordinary double point to a $\mathbb{P}^1$ in the other family. 

\section{Application: Nested reflexive polytopes}

As an application, we will show how our main result, Theorem \ref{goodness}, can be used to construct an \emph{extremal transition} between families of smooth CY threefolds, given any proper inclusion $\Delta_1 \subseteq \Delta_2$ of reflexive 4-polytopes.  

If $\mathcal{X}_1$ and $\mathcal{X}_2$ are families of smooth projective CY varieties, then an extremal transition between $\mathcal{X}_1$ and $\mathcal{X}_2$ is a degeneration of $\mathcal{X}_1$ to a singular family $\mathcal{X}_0$, followed by a resolution of $\mathcal{X}_0$ to $\mathcal{X}_2$.  Extremal transitions were originally defined in the paper \cite{morrison}, which also describes how they may be constructed from pairs of reflexive polytopes $\Delta_1 \subseteq \Delta_2$.  The papers \cite{web, acjm} discuss the application of this idea to connecting the moduli spaces of different CY families.

Let $\mathcal{X}_1$ and $\mathcal{X}_2$ be Batyrev hypersurface families associated to $\Delta_1$ and $\Delta_2$, respectively.  We have a reverse inclusion $\Delta^*_2 \subseteq \Delta^*_1$ of dual reflexive polytopes. The monomials in the defining equation of the family $\mathcal{X}_i$ are in bijective correspondence with lattice points in $\Delta^*_i$.  We can obtain a degeneration of $\mathcal{X}_1$ by allowing coefficients on monomials associated to the lattice points in $\Delta^*_1 \backslash \Delta^*_2$ to approach zero.  This results in a singular family $\mathcal{X}_0$ which is birational to $\mathcal{X}_2$.  If we can find a resolution of singularities $\mathcal{X}_2 \rightarrow \mathcal{X}_0$, then we will have constructed an extremal transition.

In the standard Batyrev construction, the families $\mathcal{X}_1$ and $\mathcal{X}_2$ live in toric varieties $X(\Sigma_i)$, where $\Sigma_i$ are MPCP resolutions of the fans $\Sigma(\Delta_i)$.  This means that to obtain a resolution $\mathcal{X}_2 \rightarrow \mathcal{X}_0$ via a toric morphism $X(\Sigma_2) \rightarrow X(\Sigma_1)$, we need MPCP resolutions $\Sigma_i$ such that $\Sigma_2$ is a subdivision of $\Sigma_1$.  However, for an arbitrary pair of reflexive 4-polytopes $\Delta_1 \subseteq \Delta_2$, the needed MPCP resolutions can fail to exist.  The paper \cite{fred3} gives a counterexample using the reflexive polytope from Section 5.

Thus, to find a toric resolution in all cases, we need to consider $\Delta$-maximal fans rather than just MPCP resolutions.  The following result shows that projective $\Delta$-maximal fans suffice to solve the problem for any pair $\Delta_1 \subseteq \Delta_2$.

\begin{lemma} Let $\Delta_1 \subseteq \Delta_2$ be reflexive 4-polytopes in $N_\mathbb{R}$, and let $\Sigma_1$ be any projective $\Delta_1$-maximal fan.  Then there exists a projective $\Delta_2$-maximal fan $\Sigma_2$ such that $\Sigma_2$ refines $\Sigma_1$. 
\end{lemma}

\begin{proof} Let $\Sigma_{\rm{GKZ}}$ be the secondary fan associated to $\nu_1, \dots, \nu_r$, the elements of $\partial \Delta_2 \cap N_\mathbb{R}$.  Then $\Gamma_{\Sigma_1, \emptyset} \in \Sigma_{\rm{GKZ}}$ is contained in the moving cone of $\Sigma_{\rm{GKZ}}$.  There must be at least one maximal cone $\sigma \in \Sigma_{\rm{GKZ}}$ such that $\Gamma_{\Sigma_1, \emptyset} \subseteq \sigma$ and $\sigma$ is also contained in the moving cone of $\Sigma_{\rm{GKZ}}$.  Then $\sigma$ is associated to a $\Delta_2$-maximal fan $\Sigma_2$ which has the needed properties.\end{proof}

As a result of the lemma, we have:

\begin{corollary} Let $\Delta_1 \subseteq \Delta_2$ be a proper inclusion of reflexive 4-polytopes, and let $\Sigma_1$ be any projective $\Delta_1$-maximal fan.  Then there is at least one projective $\Delta_2$-maximal fan $\Sigma_2$ such that an extremal transition between $\mathcal{X}(\Sigma_1)$ and $\mathcal{X}(\Sigma_2)$ exists. \end{corollary}

The results of \cite{ks} show that any pair of reflexive 4-polytopes can be connected by a chain of inclusions.  Together with the corollary, this implies that given any pair of reflexive 4-polytopes $\Delta_1$ and $\Delta_2$, we can find projective $\Delta_i$-maximal fans $\Sigma_i$ such that the families $\mathcal{X}(\Sigma_1)$ and $\mathcal{X}(\Sigma_2)$ can be connected by a finite sequence of extremal transitions between families of smooth projective CY threefolds.

\bibliography{myrefs}

\begin{thebibliography}{10}

\bibitem{acjm}
A.~C. {Avram}, P.~{Candelas}, D.~{Jancic}, and M.~{Mandelberg}.
\newblock {On the connectedness of the moduli space of Calabi-Yau manifolds}.
\newblock {\em Nuclear Physics B}, 465:458--472, February 1996.

\bibitem{Ba}
V.~V. {Batyrev}.
\newblock {Dual Polyhedra and Mirror Symmetry for Calabi-Yau Hypersurfaces in
  Toric Varieties}.
\newblock {\em {J. Alg. Geom.}}, 3:493--535, 1994.

\bibitem{BB}
V.V. Batyrev and L.A. Borisov.
\newblock {On Calabi-Yau complete intersections in toric varieties.}
\newblock {Andreatta, Marco (ed.) et al., Higher dimensional complex varieties.
  Proceedings of the international conference, Trento, Italy, June 15--24,
  1994. Berlin: Walter de Gruyter. 39--65}, 1996.

\bibitem{Borisov}
L.A. Borisov.
\newblock {Towards the Mirror Symmetry for Calabi-Yau Complete intersections in
  Gorenstein Toric Fano Varieties.}
\newblock {\texttt{arXiv:alg-geom/9310001}}, 1993.

\bibitem{coxkatz}
D.A. Cox and S.~Katz.
\newblock {\em {Mirror symmetry and algebraic geometry.}}
\newblock {Mathematical Surveys and Monographs. 68. Providence, RI: American
  Mathematical Society (AMS)}, 1999.

\bibitem{CLS}
D.A. Cox, J.B. Little, and H.K. Schenck.
\newblock {\em {Toric varieties.}}
\newblock {Graduate Studies in Mathematics 124. Providence, RI: American
  Mathematical Society (AMS)}, 2011.

\bibitem{fred3}
K.~Fredrickson.
\newblock {Extremal transitions from nested reflexive polytopes}.
\newblock 2014.
\newblock To appear in {\it Commun. Math. Phys.} \texttt{arxiv:1402.4785}.

\bibitem{fulton}
W.~Fulton.
\newblock {\em {Introduction to toric varieties. The 1989 William H. Roever
  lectures in geometry.}}
\newblock {Annals of Mathematics Studies. 131. Princeton, NJ: Princeton
  University Press}, 1993.

\bibitem{GKZbook}
I.M. Gelfand, M.M. Kapranov, and A.V. Zelevinsky.
\newblock {\em {Discriminants, Resultants, and Multidimensional Determinants.}}
\newblock {Boston, MA: Birkh\"auser}, 1994.

\bibitem{ks}
M.~Kreuzer and H.~Skarke.
\newblock {Complete classification of reflexive polyhedra in four dimensions.}
\newblock {\em Adv. Theor. Math. Phys.}, 4(6):1209--1230, 2000.

\bibitem{morrison}
D.R. Morrison.
\newblock {Through the Looking Glass}.
\newblock In {\em Mirror Symmetry III}, pages 263--277. American Mathematical
  Society and International Press, 1999.

\bibitem{Nil05}
B.~Nill.
\newblock Gorenstein toric {F}ano varieties.
\newblock {\em Manuscripta Math.}, 116(2):183--210, 2005.

\bibitem{odapark}
T.~Oda and H.~Park.
\newblock {Linear Gale Transforms and Gel'fand-Kapranov-Zelevinskij
  Decompositions}.
\newblock {\em Tohoku Math J.}, 43(3):375--399, 1991.

\bibitem{web}
C.~{Ti-Ming}, B.~R. {Greene}, M.~{Gross}, and Y.~{Kanter}.
\newblock {Black hole condensation and the web of Calabi-Yau manifolds.}
\newblock {\em Nuclear Physics B Proceedings Supplements}, 46:82--95, March
  1996.

\end{thebibliography}
\bibliographystyle{plain}
\end{document}